\documentclass[11pt]{article}
\usepackage{amscd, amsmath, amssymb, amsthm}
\usepackage[all,cmtip]{xy}
\usepackage[pagebackref]{hyperref}
\usepackage{tikz}

\title{Endomorphisms of singular del Pezzo surfaces }
\author{Burt Totaro}
\date{  }

\def\C{\text{\bf C}}
\def\P{\text{\bf P}}

\begin{document}
\maketitle
\newtheorem{theorem}{Theorem}[section]
\newtheorem{corollary}[theorem]{Corollary}
\newtheorem{lemma}[theorem]{Lemma}

\theoremstyle{definition}
\newtheorem{definition}[theorem]{Definition}
\newtheorem{example}[theorem]{Example}

\theoremstyle{remark}
\newtheorem{remark}[theorem]{Remark}

A natural problem of algebraic dynamics is to classify the normal
complex projective
varieties that admit an endomorphism of degree greater than 1.
In dimension 2, Nakayama made this classification
except among del Pezzo surfaces with Picard number 1
and quotient singularities (or equivalently, Kawamata log terminal
singularities)
\cite[Theorem 1.3]{Nakayama}.

For del Pezzo surfaces with Picard number 1 and canonical singularities
(or equivalently, du Val singularities),
Joshi worked out the classification in all cases except one:
the del Pezzo surface with a du Val singularity
of type $E_8$ and no nonzero vector fields, 
a surface which bears a curious resemblance to $\P^2$
\cite[Theorem 1.2]{Joshi}.
(For example, despite its singular point, it has the same
integral cohomology ring as $\P^2$. See section \ref{globalsect}
for the equation of this surface.)
The $E_8$ surface also arose as a possible
exceptional case in Ou's classification of contractions of hyperk\"ahler
4-folds, before it was excluded by Huybrechts and Xu \cite{Ou, HX}.

We now resolve the last case of canonical del Pezzo surfaces
with Picard number 1. Namely, the $E_8$ surface $X$
with no nonzero vector fields does not have an endomorphism
of degree greater than 1 (Theorem \ref{e8endo}). 
Our method also answers another
question which remained open for this surface: $X$ is not an image
of a proper toric variety by any morphism. After
work by Joshi and Gurjar--Pradeep--Zhang \cite{Joshi, Gurjar},
this completes the proof of the following results:

\begin{theorem}
\label{canonical}
Every canonical del Pezzo surface with Picard number 1
that admits an endomorphism of degree greater than 1
is the quotient of a projective toric surface by a finite group
that acts freely in codimension 1 and preserves the open torus orbit.
\end{theorem}

\begin{theorem}
\label{image}
Every normal Gorenstein projective surface that is an image
of $\P^2$ is isomorphic to $\P^2/G$ for some finite group $G$,
not necessarily acting freely in codimension 1.
\end{theorem}

In many cases, one can show that a variety has no nontrivial endomorphism
by showing that it (or some related variety) fails to satisfy
Bott vanishing \cite{KT};
but that seems not to help
for the $E_8$ surface $X$ with no nonzero vector fields.
Indeed, $X$ satisfies Bott vanishing
for ample Weil divisors, by Baker \cite{Baker};
its smooth locus is simply connected (so we cannot relate the problem
to a covering); and $X$ has no nontrivial contractions since
it has Picard number 1. (The other del Pezzo surface
with a du Val singularity of type $E_8$,
with a nonzero vector field, was known
not to have an endomorphism of degree greater than 1,
because it does not satisfy Bott vanishing \cite[Theorem 2.4]{Joshi}.)

Instead, we extend the method of Amerik--Rovinsky--Van de Ven,
which yields Chern number inequalities for a variety
with an endomorphism, to Deligne-Mumford stacks \cite{ARV}.
The method requires
precise estimates on global generation for the cotangent bundle twisted
by an ample line bundle. It should be useful for other hard cases
in the classification of varieties with
endomorphisms.

This work was supported by NSF grant DMS-2054553.
Thanks to Rohan Joshi for useful conversations.

\section{Notation}

Let $M$ be an algebraic stack with finite stabilizers $I_M\to M$,
so that $M$ has a coarse moduli space $M\to X$ by Keel and Mori
\cite[Tag 0DUT]{Stacks}. We say that a line bundle $L$ on $M$
is {\it ample }if some positive power of $L$ is pulled back
from an ample line bundle on $X$. 
Some authors have proposed stronger definitions
of ampleness on stacks \cite{Alper, BOW}, but this notion is enough
for our applications.

The paper mostly works over $\C$. In this context,
the classification of canonical del Pezzo surfaces with Picard number 1
was worked out by Furushima, Miyanishi--Zhang, and Ye
\cite{Furushima, MZ, Ye}.

By definition, a {\it Fano variety }is a klt projective
variety $X$ with ample anticanonical class. A del Pezzo surface
is a Fano variety of dimension 2.

\section{Chern number inequalities}

Amerik, Rovinsky, and Van de Ven showed that the only smooth
Fano 3-fold with Picard number 1 that admits an endomorphism
of degree greater than 1 is $\P^3$ \cite{ARV}. Their method
was to prove certain Chern number inequalities for a smooth variety
with an endomorphism. In some cases, this method does not suffice,
and so they also needed Amerik's analysis of endomorphisms of $X$
in terms of rational curves on $X$ \cite{Amerik}.

We extend ARV's method to smooth stacks here. The method requires
explicit information about global generation of twists of the cotangent
bundle by ample line bundles. For the $E_8$ del Pezzo surface
with no nonzero vector fields,
we will analyze the latter problem in Lemma \ref{global}.

\begin{lemma}
\label{isolated}
Let $X$ and $Y$ be Deligne-Mumford stacks (or varieties)
over a field $k$ of characteristic zero. Assume
that $X$ and $Y$ are smooth and proper over $k$,
and let $f\colon X\to Y$ be a quasi-finite morphism over $k$.
Let $L$ be a line bundle on $Y$ such that $\Omega^1_Y(L)$
is globally generated outside a zero-dimensional closed subset of $Y$.
Then, for a general section $s\in H^0(Y,\Omega^1_Y(L))$,
its pullback $t$ in $H^0(X,\Omega^1_X(f^*L))$
has zero-dimensional zero set.
\end{lemma}

Throughout, we write $\Omega^1_X$ for the differentials of $X$
over $k$, $\Omega^1_{X/k}$.

\begin{proof}
For smooth varieties rather than smooth stacks, this is essentially
\cite[Lemma 1.1]{ARV}. The same proof works, as follows. To begin,
the assumption on global generation implies that a general section
$s\in H^0(Y,\Omega^1_Y(L))$ vanishes only on a zero-dimensional
closed subset of $Y$, by an easy Bertini-type
theorem \cite[Remark 6]{Kleiman}.

We first give the proof
when $X$ and $Y$ have dimension 2, the case needed in this paper.
On \'etale coordinate charts for $X$ and $Y$, the map
$\varphi\colon s\mapsto t=f^*(s)$ is locally given by the matrix
of derivatives of $f$. So if, for a general section $s$,
its pullback $f^*(s)$ vanishes on a curve (meaning an
integral closed substack of dimension 1), then this curve
must be an irreducible component $V$ of the ramification divisor
of $f$ (the locus where $f$ is not \'etale, which is not all of $X$
because $k$ has characteristic zero).
In this case, $f^*(s)$ must vanish on $V$ for all
$s\in H^0(Y,\Omega^1_Y(L))$.

Since $f$ is quasi-finite, it is not constant on the curve $V$,
and so the derivative of $f\colon V\to Y$ is generically nonzero
(using again that $k$ has characteristic zero). Equivalently,
the kernel of 
$$f^*(\Omega^1_Y(L))|_{V}\to \Omega^1_X(f^*L)|_{V}$$
generically has rank 1. By the previous paragraph, all
global sections of $\Omega^1_Y(L)$ restricted to the curve $V$
lie in this rank-1 subbundle. This contradicts our assumption
that $\Omega^1_Y(L)$ is globally generated outside a
zero-dimensional closed subset.

Following Amerik--Rovinsky--Van de Ven, this proof generalizes
to arbitrary dimension as follows. Let $X$ and $Y$ have dimension $m$.
Let $R_i$ be the locus in $X$ where the derivative of $f$ has rank
at most $i$. Since $f$ is quasi-finite and $k$ has characteristic zero,
$R_i$ has dimension $\leq i$ for each $i$.
Now suppose that for a general section
$s\in H^0(X,\Omega^1_X(L))$, the pullback section $t=f^*(s)$ vanishes
along a curve $C_s$. Let $j$ be such that $C_s\subset R_j$
and $C_s\not\subset R_{j-1}$ for a general $s$. Let $V$ be an
irreducible component of $R_j$ containing $C_s$, for a general $s$.

The section $s$ induces a section $s_V$ of the sheaf
$$A_j = f^*(\Omega^1_Y(L))|_V/B_j,$$
where
$$B_j=\ker(f^*(\Omega^1_Y(L))|_V\to \Omega^1_X(f^*L)|_V).$$
Clearly the section $s_V$ vanishes on $C_s$.

The sheaf $s_j$ on $V\subset R_j$ has rank $j$, it is locally free
outside $R_{j-1}$, and it is generated by the sections $s_V$
outside a finite set (for $s\in H^0(Y,\Omega^1_Y(L))$). By the easy
Bertini-type theorem above, outside of $R_{j-1}$ the general $s_V$
vanishes only on a finite set, so $C_s\subset R_{j-1}$,
a contradiction.
\end{proof}

As in Amerik--Rovinsky--Van de Ven (for smooth varieties
rather than smooth stacks),
we have the consequence:

\begin{corollary}
\label{hurwitz}
(the ``Hurwitz formula'') In the situation of Lemma
\ref{isolated}, with $X$ and $Y$ of dimension $n$,
$$\deg(f)\; c_n(\Omega^1_Y(L))\leq c_n(\Omega^1_X(f^*L)).$$
\end{corollary}

\begin{proof}
Let $s$ be a general section of $\Omega^1_Y(L)$. By Lemma
\ref{isolated}, both $s$ and $f^*(s)\in H^0(X,\Omega^1_X(f^*L))$
have zero-dimensional zero sets.
Clearly the substack $\{f^*(s)=0\}$ in $X$
contains the pullback of the substack $\{s=0\}$ in $Y$, which is a
local complete intersection
substack of $Y$. It follows
that the degree of $\{f^*(s)=0\}$ in $X$ is least the degree
of that pullback, which is $\deg(f)$ times the degree of $\{s=0\}$ in $Y$.
(These degrees are rational numbers, since $X$ and $Y$
are Deligne-Mumford stacks.) The top Chern class of a vector bundle
is represented by the cycle associated to the zeros of any section
for which the zero set has the expected dimension.
That proves the desired inequality.
\end{proof}

\begin{corollary}
\label{inequality}
Let $X$ be a Deligne-Mumford stack (or variety) of dimension $n>0$
over a field $k$
of characteristic zero. Assume that $X$ is smooth and proper over $k$.
Let $L$ be an ample line bundle on $X$
such that $\Omega^1_X(L)$ is globally
generated outside a finite set. Suppose that
$$L^{n-1}c_1(\Omega^1_X)+L^{n-2}c_2(\Omega^1_X)+\cdots+c_n(\Omega^1_X)>0,$$
or that this expression is $\geq 0$ and $X$ is Fano.
Then $X$ has no endomorphism $f$ such that
$f^*(L)$ is numerically equivalent to $aL$ with $a>1$. 
\end{corollary}

In the inequality, we identify a line bundle with a Cartier divisor
up to linear equivalence, via the first Chern class.

\begin{proof}
Suppose that $X$ has an endomorphism $f$ with $f^*(L)\equiv aL$
for some $a>1$. Since $L^n>0$, $f$ has degree $a^n$,
which is greater than 1. Applying Corollary \ref{hurwitz} to the iterate $f^m$
for any $m\geq 0$, we have
$$a^{mn}c_n(\Omega^1_X(L))\leq c_n(\Omega^1_X(a^m L)),$$
that is,
\begin{multline*}
a^{mn}(L^n+L^{n-1}c_1(\Omega^1_X)+\cdots+c_n(\Omega^1_X))\\
\leq a^{mn}L^n + a^{m(n-1)}L^{n-1}c_1(\Omega^1_X)+\cdots+c_n(\Omega^1_X).
\end{multline*}
Taking $m$ to infinity, it follows that
$$L^{n-1}c_1(\Omega^1_X)+\cdots+c_n(\Omega^1_X)\leq 0,$$
as we want.

Next, suppose that this last inequality is an equality. Then,
again taking $m$ to infinity, it follows that $L^{n-1}c_1(\Omega^1_X)\geq 0$,
that is, $L^{n-1}K_X\geq 0$ (where $K_X$ is the canonical line bundle).
In particular, this fails if $X$ is Fano.
So, if $X$ is Fano, we have in fact
$$L^{n-1}c_1(\Omega^1_X)+\cdots+c_n(\Omega^1_X)<0.$$
\end{proof}

\section{Morphisms}

Here we apply Corollary \ref{hurwitz} to some related problems:
bounding the degree of morphisms to a given variety,
and showing that a given variety is not the image
of a toric variety.

\begin{theorem}
\label{morphisms}
Let $X$ and $Y$ be Deligne-Mumford stacks (or varieties) of dimension $n>0$
over a field $k$
of characteristic zero. Assume that $X$ and $Y$ are smooth and proper over $k$,
with Picard number 1.
Let $L$ be an ample line bundle on $Y$ such that $\Omega^1_Y(L)$ is globally
generated outside a finite set. Suppose that
$$L^{n-1}c_1(\Omega^1_Y)+L^{n-2}c_2(\Omega^1_Y)+\cdots+c_n(\Omega^1_Y)>0,$$
or that this expression is $\geq 0$ and $X$ is Fano.
Then there is an upper bound on the degrees of all morphisms
from $X$ to $Y$.
\end{theorem}

\begin{proof}
Let $H_X$ be an ample line bundle on $X$. For any morphism
$f\colon X\to Y$, we must have the numerical equivalence
$f^*L\equiv mH_X$ for some nonnegative
rational number $m$. Then $f$ has degree $m^n(H_X^n/L^n)$.
So if there are morphisms $f$ of arbitrarily high degree,
then these morphisms have $m$ arbitrarily large. For $m>0$,
$f^*L$ is ample, and so $f$ contracts no curves; that is, $f$ is finite.

For such a morphism $f$, Corollary \ref{hurwitz} gives that
\begin{multline*}
m^n(H_X^n/L^n)(L^n+L^{n-1}c_1(\Omega^1_Y)+\cdots+c_n(\Omega^1_Y))\\
\leq m^nH_X^n+m^{n-1}H_X^{n-1}c_1(\Omega^1_X)+\cdots+c_n(\Omega^1_X).
\end{multline*}
Since $m$ can be arbitrarily large, looking at the coefficient of $m^n$
gives that 
$$L^{n-1}c_1(\Omega^1_Y)+\cdots+c_n(\Omega^1_Y)\leq 0,$$
as we want. If equality holds, then looking at the coefficient of $m^{n-1}$
gives that $H_X^{n-1}c_1(\Omega^1_X)\geq 0$, that is,
$H_X^{n-1}K_X\geq 0$. That cannot occur if $X$ is Fano, as we want.
\end{proof}

One would expect that the coarse moduli space of $Y$
is not an image of any proper toric
variety under the assumption of Theorem \ref{morphisms}. For smooth varieties
rather than smooth stacks, Occhetta and Wi\'{s}niewski showed much more:
the only smooth complex projective variety $Y$ with Picard number 1
that is an image of a proper toric variety is $\P^n$
\cite[Theorem 1.1]{Occhetta}. For varieties $Y$
with klt singularities, we have only the following
result in dimension 2.

\begin{theorem}
\label{toric}
Let $Y$ be a klt projective surface
over a field $k$
of characteristic zero. Then $Y$ has quotient singularities;
let $M$ be the corresponding smooth Deligne-Mumford stack.
Let $L$ be an ample line bundle on $M$ such that $\Omega^1_M(L)$ is globally
generated outside a finite set. Suppose that
$$L\cdot c_1(\Omega^1_M)+c_2(\Omega^1_M)\geq 0.$$
Then there is no surjective morphism from a proper toric variety
to $Y$.
\end{theorem}

\begin{proof}
Since $Y$ is a klt surface, it has quotient singularities.
Suppose that there is a surjective morphism from a proper
toric variety $X$ to $Y$. Let $X\to X_1\to Y$ be the Stein
factorization, so $X\to X_1$ is a contraction and $X_1\to Y$
is finite. Every contraction of a toric variety
is toric \cite[Proposition 2.7]{Tanakatoric}. (This was known
earlier for projective toric varieties \cite[Theorem 6.28
and exercise 7.2.3]{CLS}.) So, replacing $X$ by $X_1$,
we can assume that the surjection $f\colon X\to Y$
is finite.

Then $X$ is a toric surface, and so it also has quotient singularities.
Let $N$ be the corresponding smooth Deligne-Mumford stack
\cite[Proposition 2.8]{Vistoli}.
We have coarse moduli spaces $N\to X$ and $M\to Y$,
and $N$ and $M$ have trivial stabilizer in codimension 1.
Since $N$ is smooth, the proof of Theorem \ref{e8endo} shows that
the finite morphism $f\colon X\to Y$ lifts to a quasi-finite
morphism $g\colon N\to M$. Then $H_N:=g^*L$ is an ample
line bundle on $N$.

For each positive integer $m$, the toric variety $X$
has a ``multiplication'' endomorphism $e_m$, extending the morphism
$x\mapsto x^m$ on the open torus in $X$. Since this morphism is finite,
the same argument shows that $e_m$ lifts to an endomorphism
of the stack $N$, which we also call $e_m$. Then $e_m^*(H_N)\cong
mH_N$. So the morphism $g\circ e_m\colon N\to M$ satisfies
$(g\circ e_m)^*(L)\cong mH_N$. Apply Corollary \ref{hurwitz}
to these morphisms with $m$ arbitrarily large.
As in the proof of Theorem \ref{morphisms}, it follows
that
$$L\cdot c_1(\Omega^1_M)+c_2(\Omega^1_M)\leq 0,$$
and if equality holds, then $H_N\cdot K_N\geq 0$.
But every toric variety $X$ has $-K_X$ big, and hence
$-K_N$ is big since $N\to X$ is an isomorphism in codimension 1.
It follows that $H_N\cdot K_N<0$. So in fact we must have
$$L\cdot c_1(\Omega^1_M)+c_2(\Omega^1_M)< 0.$$
\end{proof}

\section{Global generation for the $E_8$ del Pezzo surface}
\label{globalsect}

\begin{lemma}
\label{global}
Let $X$ be the del Pezzo surface over $\C$
with a du Val singularity of type $E_8$ and no nonzero vector fields.
Write $O_X(1)$ for the anticanonical line bundle $-K_X$,
which is the ample generator
of the Picard group of $X$. Then the reflexive sheaf $\Omega^{[1]}_X(2)$
is globally generated outside finitely many points. Equivalently,
if $M$ denotes the corresponding smooth stack, the vector bundle
$\Omega^1_M(2)$ is globally generated outside finitely many points.
\end{lemma}

\begin{proof}
Explicitly, $X$ is the sextic surface $\{w^2+z^3+x^5y+x^4z=0\}$
in the weighted projective plane
$W=\P^2(1,1,2,3)$ \cite{Gurjar}.
Here $X$ misses the singular points of $W$,
but it has an $E_8$ singularity at the point $q=[0,1,0,0]$.
(There are two isomorphism classes of del Pezzo surfaces with
a du Val singularity of type $E_8$,
the other being the surface $\{w^2+z^3+x^5y=0\}$
in $W$. That one is a degeneration of $X$ and has a nonzero vector field.
Another way to describe the difference is that our surface $X$
has $H^1(X,TX)=0$;
that is, every locally trivial
first-order deformation of $X$ is trivial.)

Let $M$ be the corresponding smooth Deligne-Mumford stack over $\C$.
Thus we have a coarse moduli map $M\to X$,
and $M$ has the binary icosahedral group of order 120 as stabilizer group
at one point $q$ (mapping to $[0,1,0,0]$ in $X$).

The line bundle $-K_X$ pulls back to the line bundle $-K_M$; write
$O_X(1)$ and $O_M(1)$ for these line bundles.
From the equation for $X$, we see that the linear system $|O_X(1)|$
is a pencil of genus-1 curves through the point $p=[0,0,-1,1]$.
Here $H^0(M,O_M(1))\cong H^0(X,O_X(1))$, and so we have the same description
of the linear system $|O_M(1)|$. Likewise, $H^0(M,\Omega^1_M(2))
\cong H^0(X,\Omega^{[1]}_X(2))$, by definition of reflexive differentials,
and so it will suffice to show that $\Omega^1_M(2)$ is globally generated
outside finitely many points.

The curves in the linear system $|O_X(1)|$ are the curves of the form
$\{ax+by=0\}\cap X$. We compute that there are exactly three singular
curves in this linear system:
$\{y=\pm \sqrt{-4/27}x\}\cap X$ (both isomorphic to the nodal cubic curve) and 
$\{x=0\}\cap X$ (containing the $E_8$ singularity $q$ of $X$).
Write $F_1,F_2,F_3$
for the corresponding curves in the stack $M$, all going through
the point $p$. Thus $F_1$ and $F_2$
are isomorphic to the nodal cubic curve. The curve $F_3$ is more complicated,
since it contains the point $q$ of $M$ with nontrivial stabilizer group,
and it is also not a smooth stack at that point. The coarse moduli space
$C_3$ of $F_3$ is isomorphic to the cuspidal cubic curve. (See Figure 1.)

\begin{figure}
\centering
\begin{tikzpicture}
\draw (-1,0.2) .. controls (-0.5,0.05) .. (0,0)
.. controls (1,0) and (4,1) .. (3,1)
  .. controls (2,1) and (3,0) .. (4,0);
\node (F1) at (4.3,0) {$F_1$};
\draw (-1*0.707+0.2*0.707,1*0.707+0.2*0.707) .. controls (-0.5*0.707+0.05*0.707,0.5*0.707+0.05*0.707) .. (0,0)
.. controls (0.707,-0.707) and (5*0.707,-3*0.707) .. (4*0.707,-2*0.707)
  .. controls (3*0.707,-0.707) and (3*0.707,-3*0.707) .. 
  (4*0.707,-4*0.707);
\node (F2) at (4.3*0.707,-4.3*0.707) {$F_2$};
\draw (0.2,1) .. controls (0.05,0.5) .. (0,0)
  .. controls (-0.05,-0.5) and (-1,-3) .. (0,-3)
  .. controls (-1,-3) and (-0.7,-3.4) .. (-1,-4);
\node (F3) at (-1.1,-4.3) {$F_3$};
\node (p) at (-0.2,-0.2) {$p$};
\node (q) at (0.2,-3) {$q$};
\fill[color=black] (0,-3) circle (0.4ex);
\end{tikzpicture}
\caption{The singular curves in the elliptic pencil $|O_M(1)|$
on $M$}
\end{figure}
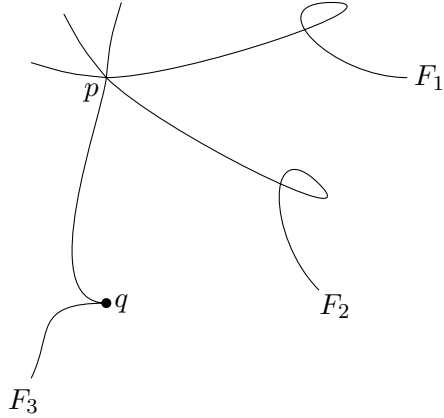

Here $M$ (or equivalently its coarse moduli space $X$) has Betti numbers
$1,0,1,0,1$, because the minimal resolution of $X$ adds
8 $(-2)$-curves and is isomorphic to a blow-up of $\P^2$ at 8 points. By the
degeneration of the Hodge spectral sequence
for $M$ \cite[Theorem 1.9]{Satriano}, it follows
that $h^0(M,\Omega^1_M)=0$, $h^1(M,\Omega^1_M)=1$, and $h^2(M,\Omega^1_M)=0$.
To analyze the sections of $\Omega^1_M(2)$, we will use the exact sequence
$$0\to \Omega^1_M(a-1)\to \Omega^1_M(a)\to \Omega^1_M(a)|_F\to 0$$
of coherent sheaves on $M$. Here $a$ is an integer and $F$ is any curve
in the linear system $|O_M(1)|$. As a first step, we will analyze
sections of $\Omega^1_M(1)$. In view of the exact sequence above, we are led
to ask about sections of $\Omega^1_M(1)|_F$.

To analyze these sections, consider the exact sequence of coherent sheaves
on $F$:
\begin{equation*}
0\to N_{F/M}^*\to \Omega^1_M|_F\to \Omega^1_F\to 0.
\tag{\textasteriskcentered}
\label{ext}
\end{equation*}
The conormal bundle $N_{F/M}^*$ is a line bundle even for $F$ singular,
namely $O_F(-1)$. The sheaf $\Omega^1_F$ is a line bundle on $F$
if $F$ is smooth, but not for the singular
curves $F$. First consider a general curve $F$ in the linear system
of $O_M(1)$; so $F$ is a smooth curve of genus 1. Then the extension 
\eqref{ext} is classified by an element $\alpha$
in $H^1(F,N_{F/M}^*\otimes TF)$.
Multiplying by $\alpha$ gives a linear map $H^0(F,N_{F/M})\to H^1(F,TF)$,
which describes how deforming $F$ inside $M$ deforms $F$ up to isomorphism.
Since the genus-1 curve $F_1$ in the pencil $|O_M(1)|$
is nodal (hence has
$j$-invariant $\infty$), the pencil is not isotrivial. Thus, for a general
curve $F$ in the pencil, the element $\alpha$ must be nonzero
(using that we are in characteristic zero).
That is, for $F$ general, the extension \eqref{ext} is nontrivial.

After tensoring with $O(1)$, the exact sequence $0\to N_{F/X}^*(1)
\to \Omega^1_M(1)|_F\to \Omega^1_F(1)\to 0$ is still non-split,
for $F$ general. Here $N_{F/X}^*(1)\cong O_F$ and $\Omega^1_F(1)\cong
O_F(1)$. (Note that $O_F(1)$ is the line bundle $O(p)$ associated
to the point $p$ in $F$.) Consider the long exact sequence of cohomology,
$$\xymatrix@=5pt{
& \C & & \C \\
0 \ar[r] & H^0(F,N_{F/M}^*(1))\ar[r]\ar@{=}[u] &
H^0(F,\Omega^1_M(1)) \ar[r] & H^0(F,\Omega^1_F(1)) \ar@{=}[u]
\ar `r[d] `[dlll] `[ddlll] [ddll]\\
& & & \\
& H^1(F,N_{F/M}^*(1))\ar[r]\ar@{=}[d] &
H^1(F,\Omega^1_M(1)) \ar[r] & H^1(F,\Omega^1_F(1)). \ar@{=}[d]\\
& \C & & 0
}$$
I claim that the map $H^0(F,\Omega^1_M(1))\to H^0(F,\Omega^1_F(1))\cong\C$
must be zero. If not, let $s$ be a section of $\Omega^1_M(1)|_F$ that maps
to the standard section of $\Omega^1_F(1)\cong O(p)$ that vanishes
exactly at $p$.
By subtracting a suitable section of the trivial line bundle $N_{F/X}^*(1)$,
we can assume that $s$ vanishes at $p$. But then $s$ can be viewed as
a splitting of the surjection $\Omega^1_M(1)|_F\to O(p)$, a contradiction.
The claim is proved.
Thus the map $H^0(F,N_{F/X}^*(1))\cong\C\to H^0(F,\Omega^1_M(1))$
is an isomorphism. By the long exact sequence above,
it follows that $H^1(F,\Omega^1_M(1))=0$.

Thus, on $M$, we have the long exact sequence:
$$\xymatrix@=5pt{
& 0 & & \C \\
& H^0(M,\Omega^1_M)\ar[r]\ar@{=}[u] &
H^0(M,\Omega^1_M(1)) \ar[r] & H^0(F,\Omega^1_M(1)) \ar@{=}[u]
\ar `r[d] `[dlll] `[ddlll] [ddll]\\
& & & \\
& H^1(M,\Omega^1_M)\ar[r]\ar@{=}[d] &
H^1(M,\Omega^1_M(1)) \ar[r] & H^1(F,\Omega^1_M(1)). \ar@{=}[d]\\
& \C & & 0
}$$
I claim that $H^0(M,\Omega^1_M(1))$ is zero. To see this,
let $s$ be a section of $\Omega^1_M(1)$. By the analysis above, on a general
curve $F$ in $|O_M(1)|$, $s$ lies in the rank-1 subbundle $N_{F/M}^*(1)$.
So that is true on all of $M-p$ (where this subbundle is defined).
Equivalently, in terms of the morphism $f\colon M-p\to \P^1$
given by $O_M(1)$, we can view
$s$ as an element of $H^0(M-p, f^*(\Omega^1_{\P^1}(1)))=H^0(M-p,O_M(-1))
=H^0(M,O_M(-1))=0$. Thus $H^0(M,\Omega^1_M(1))=0$, as claimed.
Therefore, the boundary map
in the sequence above is an isomorphism, and so also $H^1(M,\Omega^1_M(1))=0$.

We made the analysis above for a general curve $F$ in the linear system
$|O_M(1)|$. Now let
$F$ be any curve in $|O_M(1)|$, possibly singular. 
By the exact sequence
$$\xymatrix@=5pt{
H^0(M,\Omega^1_M(1))\ar[r]\ar@{=}[d]
& H^0(F,\Omega^1_M(1))\ar[r]
& H^1(M,\Omega^1_M)\ar[r]\ar@{=}[d] & H^1(M,\Omega^1_M(1)),\ar@{=}[d]\\
0 & & \C & 0
}$$
we have $h^0(F,\Omega^1_M(1))=1$. By the exact sequence
$$\xymatrix@=5pt{
0\ar[r] & H^0(F,N_{F/M}^*(1))\ar[r]\ar@{=}[d]
& H^0(F,\Omega^1_M(1))\ar[r]\ar@{=}[d]
& H^0(F,\Omega^1_{F}(1)) \ar[r]
& H^1(F,N_{F/M}^*(1)),\ar@{=}[d]\\
& \C & \C & & \C
}$$
we have $h^0(F,\Omega^1_{F}(1))\leq 1$. So the space of sections
of the sheaf $\Omega^1_{F}$
supported on a finite set has dimension at most 1 (since this property is not
affected by tensoring with a line bundle). This will be relevant
when $F$ is singular.

We now turn to the bundle $\Omega^1_M(2)$. As above, for every curve
$F$ in the elliptic pencil $|O_M(1)|$, we have an exact sequence
$$H^0(M,\Omega^1_M(1))\to H^0(M,\Omega^1_M(2))
\to H^0(F,\Omega^1_M(2))\to H^1(M,\Omega_M(1)).$$
The outer two groups are zero, and so $H^0(M,\Omega^1_M(2))$ maps
isomorphically to $H^0(F,\Omega^1_M(2))$.

To analyze the latter group, use the exact sequence of coherent sheaves
on $F$, $0\to N_{F/M}^*(2)\to \Omega^1_M(2)|_F\to \Omega^1_F(2)\to 0$.
Here $N_{F/M}^*(2)$ is isomorphic to $O_F(1)=O(p)$, which has base locus
the point $p$. For $F$ smooth, 
$\Omega^1_F(2)$ is isomorphic to $O_F(2)$, which is basepoint-free
(using that $F$ has genus 1). Also, $H^1(F,N_{F/M}^*(2))=0$,
and so we have a short exact sequence on $H^0$. 
It follows that $h^0(F,\Omega^1_M(2))=3$ and that
$\Omega^1_M(2)|_F$ is globally generated outside the point $p$.
By the previous paragraph, it follows that 
$h^0(M,\Omega^1_M(2))=3$ and that $\Omega^1_M(2)$ is globally
generated outside the three singular curves $F_1,F_2,F_3$ in the pencil.

Now let $F$ be one of the singular curves in the pencil $|O_M(1)|$.
By two paragraphs back, we know
that $H^0(F,\Omega^1_M(2))$ has dimension 3, and these sections
are restrictions of global sections of $\Omega^1_M(2)$ on $M$.
Consider the exact sequence
$$0\to N_{F/M}^*(2)\to \Omega^1_M(2)|_{F}\to \Omega^1_{F}(2)\to 0$$
of coherent sheaves on $F$. Here $N_{F/M}^*(2)$ is the line bundle 
$O(p)$ on $F$; note that $p$ is a smooth point of $F$. Let $C$ be the coarse
moduli space of $F$ (which is the same as $F$ except when $F$ is $F_3$).
Then $C$ is a curve of arithmetic genus 1, and so
$h^0(C,O(p))=1$ and $h^1(C,O(p))=0$
\cite[Tag 0E3A]{Stacks}. Since $\pi\colon F\to C$ is a good moduli space
in Alper's sense (using that we are in characteristic zero),
we have $H^j(F,\pi^*E)\cong H^j(C,E)$ for every integer $j$
and every
coherent sheaf $E$ on $C$ \cite[Definition 4.1, Proposition 4.5]{Alper}.
So $h^0(F,O(p))=1$ and $h^1(F,O(p))=0$, or equivalently
$h^0(F,N_{F/M}^*(2))=1$ and $h^1(F,N_{F/M}^*(2))=0$.
In particular, the line bundle $N_{F/M}^*(2)$
is globally generated outside finitely many points of $F$.

Since $h^1(F,N_{F/M}^*(2))=0$, our exact sequence of sheaves on $F$
gives a short exact sequence
$$0\to H^0(F,N_{F/M}^*(2))\to H^0(F,\Omega^1_M(2)|_{F})
\to H^0(F,\Omega^1_{F}(2))\to 0.$$
Therefore,
$h^0(F,\Omega^1_{F}(2))=2$. We showed earlier that the space
of sections of $\Omega^1_F$ supported on a finite set has dimension
at most 1, and so the same is true for $\Omega^1_F(2)$.
Therefore,
there is a section of $\Omega^1_{F}(2)$ that is not supported
on a finite set.
It follows that $\Omega^1_M(2)|_{F}$ is globally generated outside finitely
many points of $F$. As we said, these sections extend to $M$.
This completes the proof
that $\Omega^1_M(2)$ is globally generated outside finitely many points
of $M$.
\end{proof}

\section{Conclusions on the $E_8$ del Pezzo surface}

\begin{theorem}
\label{e8endo}
Let $X$ be the del Pezzo surface over $\C$ with a du Val singularity
of type $E_8$
and no nonzero vector fields.
Then $X$ has no endomorphism of degree greater than 1.
\end{theorem}

\begin{proof}
Suppose that there is an endomorphism $f\colon X\to X$
of degree greater than 1. Since $X$ has Picard number 1,
the pullback of the ample line bundle $O_X(1):=-K_X$ must be numerically
equivalent to $O_X(a)$ for some real number $a>1$,
which in particular is ample.
So $f$ does not contract any curves, and hence $f$ is finite.

Let $M$ be the smooth proper
Deligne-Mumford stack associated to $X$;
there is a coarse moduli map $M\to X$.
I claim that $f$ lifts to an endomorphism of $M$. To see that, consider
the composition $M\to X\xrightarrow[]{f} X$. Clearly this lifts uniquely
to $M$ except possibly at the finitely many points $z$ of $M$ that are mapped
to the singular point $q\in X$. The point is that $M$ is smooth of complex
dimension 2. As a result,
for a small ball $B$ in a coordinate chart around $z$,
$B-z$ is simply
connected. So, if $N$ denotes a small ball around $q$ in $X$,
the map from $B-z$ to $N-q$ lifts to the universal cover of $N-q$
(a $G$-covering, where $G$ is the binary icosahedral group), and that is
a coordinate chart for $q$ in $M$, as we want. The resulting endomorphism
of $M$ is quasi-finite.

Next, the Chern number $c_2(\Omega^1_M)$ is the topological
Euler characteristic of $M$ (or $X$),
namely 3, minus $119/120$, because of the stabilizer
group $G$ of order 120 at the point $q$ \cite[Theorem 7.3]{Blache}.
That is, $c_2(\Omega^1_M)=241/120$.

Let $L=O_M(2)$. By Lemma \ref{global},
$\Omega^1_M(L)$ is globally generated outside a finite set.
We have $O_M(1)^2=1$, and $c_1(\Omega^1_M)=O_M(-1)$. Therefore,
\begin{align*}
L\cdot c_1(\Omega^1_M)+c_2(\Omega^1_M)&=-2\; O_M(1)^2+241/120\\
&=1/120,\\
\end{align*}
which is (barely) positive.
By Corollary \ref{inequality}, it follows that there is
no endomorphism $f\colon M\to M$ with $f^*O(1)\equiv O(a)$
for some $a>1$. 
\end{proof}

By Theorem \ref{morphisms},
it also follows that the $E_8$ surface $X$
with no nonzero vector fields
does not have morphisms of unbounded degree from any fixed
smooth projective surface with Picard number 1.
Finally, by Theorem \ref{toric},
$X$ is not the image of a proper toric variety by any morphism.
By the work of Joshi and Gurjar--Pradeep--Zhang
\cite{Joshi, Gurjar}, we have completed the proof
of Theorems \ref{canonical} and \ref{image}.


\small \sc UCLA Mathematics Department, Box 951555,
Los Angeles, CA 90095-1555

totaro@math.ucla.edu
\end{document}